\documentclass[12pt, reqno]{amsart}
\usepackage{ amsmath,amsthm, amscd, amsfonts, amssymb, graphicx, color}
\usepackage[bookmarksnumbered, colorlinks, plainpages]{hyperref}
\newtheorem{thm}{Theorem}[section]
\newtheorem{cor}[thm]{Corollary}
\newtheorem{lem}[thm]{Lemma}
\newtheorem{prop}[thm]{Proposition}
\newtheorem{defn}[thm]{Definition}

\newtheorem{exam}[thm]{Example}
\numberwithin{equation}{section}

\begin{document}
\title[Refinement rings, localization and ... ]{Refinement rings, localization and diagonalizability of matrices}
\author[N. Ashrafi]{NAHID. ASHRAFI}
\address{
Faculty of Mathematics\\Statistics and Computer Science\\  Semnan
University\\ Semnan, Iran\\
} \email{<nashrafi@semnan.ac.ir>}%
\author[R. Bahmani]{RAHMAN. BAHMANI SANGESARI}
\address{
Faculty of Mathematics\\Statistics and Computer Science\\  Semnan
University\\ Semnan, Iran\\
} \email{<rbahmani@semnan.ac.ir> }%
\author[M. Sheibani]{MARJAN. SHEIBANI ABDOLYOUSEFI}
\address{
Faculty of Mathematics\\Statistics and Computer Science\\  Semnan University\\ Semnan, Iran\\
} \email{<m.sheibani1@gmail.com>}%
\begin{abstract}
In this paper we prove that if $R$ is a commutative refinement ring and $M,N$ are two $R$-modules then, $M\cong N$
if and only if for every maximal ideal $m$ of $R$, $M_m\cong N_m$. We prove if $R$ is a refinement ring,then every
regular matrix over $\frac{R}{J(R)}$ admits a diagonal reduction iff every regular matrix over $R$ admits a diagonal
 reduction.

\end{abstract}
\maketitle

\section{Introduction}
All rings considered below are associative with unit and all modules are unital. Recall that a ring $R$ is called B$\acute{e}$zout if every finitely generated ideal of $R$ is principal. An $m\times n$ matrix $A$ over a ring $R$ admits a
diagonal reduction if there exist invertible matrices $P$ and $Q$
such that $PAQ$ is a diagonal matrix, where by the diagonal
matrix, we mean a matrix $(a_{ij})_{m\times n}$, such that
$a_{ij}=0$ for all $i\neq j$. Following Kaplansky \cite{kap}, a ring $R$ is called a right (left) Hermite ring if every $1 \times2$ $ ( 2 \times 1 )$ matrix  over $R$ admits a digonal reduction. He also called a ring $R$ to be an elementary divisor ring provided that every $ m\times n $ matrix over $ R $ is equivalent to a diagonal matrix, $ diag (d_{1}, d_{2}, ....,d_{m}) $, where $ d_{i} $ is a total divisor of $ d_{i + 1 } $ $(d_{i + 1} R d_{i + 1} \subseteq d_{i} R \cap R d_{i})$. The study of diagonalizability of matrices over rings has a rich history. Before Kaplansky's work on elementary divisor rings in 1948 {\cite{kap}, many authors like Smith \cite{Sm}, Dickson \cite{dick}, Wedderborn \cite{wedder}, Warden and Jacobson investigated this, over any commutative and non-commutative Euclidean domains and commutative principal ideal domains. Henriksen \cite{hen}, has proved that every unit regular ring is an elementary divisor ring and Levy \cite{levy}, has shown that every square matrix over any serial ring admits diagonal reduction. Menal and Moncasi \cite{menal}, answered the question of diagonalizability of rectangular matrices over any regular ring by the cancellation law over the monoid of finitely generated projective modules. They proved that every rectangular matrices over given regular ring $R$ admits a diagonal reduction if and only if the finitely generated projective $R$-modules satisfy the following cancellation law:$$2R\oplus A\cong R\oplus B\Longrightarrow R\oplus A\cong B,$$for all finitely generated projective $R$- modules $A, B$.\\
In $1997$ Ara, Goodearl, O'mera and Pardo \cite{ara}, extended that from
regular rings to exchange rings and showed that every regular
matrix over an exchange ring $R$ admits a diagonal reduction if
and only if $2R\oplus A\cong R\oplus B$ implies that $ R\oplus
A\cong B$ for all finitely generated projective $R$-modules $A$ and $B$.\\
Following Chen \cite{Chen} a ring $R$ is said to be an exchange ring if for any right $R$-module $M$ and any two decompositions $M = A\oplus B = \bigoplus_{i \in I} A_{i}$, where $A_{R} \cong R$ and index set $I$ is finite,  there exist $A^{\prime}_{i} \subseteq A_{i}$ such that $M = A \oplus (\bigoplus_{i \in I} A^{\prime}_{i})$. As some known classes of rings, for example polynomial rings over the ring of integer numbers is not  exchange, we are interested to investigate the diagonalizability of matrices over wider classes of rings, that is called refinement rings and contain such rings. \\
In section 2 we study the localization of refinement rings. We prove that if $R$ is a commutative refinement ring and $M, N$ be two $R$- modules then $M\cong N$ iff $M_m \cong N_m$ for all maximal ideal $m$ of $R$. In section 3 we investigate some properties of Hermite rings. We explore it over power series and polynomial rings over Hermite ring. We construct an example of extension ring of a Hermite ring that is not  Hermite. We also make an example which shows that the tensor product of two Hermite algebras is not Hermite. Then we extend the result of \cite{ara}, from exchange rings to refinement rings. We show that over a refinement ring $R$   every regular matrix admits diagonal reduction iff every regular matrix over $\frac{R}{J(R)}$ admits a diagonal reduction.\\
Throughout this paper, ideals are two
sided ideals and modules are right $R$-modules. We also use
$M_n(R)$ for the ring of $n\times n$ matrices over $R$ with
identity $I_n$, $GL_n(R)$ the invertible $n\times n$ matrices over
$R$ and $FP(R)$ the class of finitely generated projective
$R$-modules and $V(R)$ the monoid of finitely generated projective $R$-modules.

\section{Refinement Rings}

Dubbertin \cite{Dub} in 1982 defined the monoid $(M,+,0)$  to be a refinement monoid if the following conditions are satisfied :\\ (1) There are no non-zero inverse elements, i.e, if x + y = 0 then x = y = 0. \\ (2) M has the refinement property, that is, given $ x_{i}, y_{j} \in M$ with $\sum_{i} x_{i} = \sum_j y_{j}$, there are $z_{ij} \in M$   $ ( i < n, j< m $, where $ n, m \in \Bbb{N}$ and $  n, m \geqslant 2)$ such that $x_{i} = \sum_j z_{ij}$ and $y_{j} = \sum_i z_{ij}$.\\Note that we need only to show the above property for $m=n=2$.
\begin{defn}
We say that a ring $R$ is a refinemet ring if the monoid of finitely generated projective $R$-modules, $V(R)$, has refinement property.
\end{defn}
In 1964, Crawley and Jonsson \cite{CJ}, proved that the monoid of finitely generated projective modules of every
exchange ring, has the refinement property so every exchange ring is a refinement ring but the converse is not true, as
 we see the ring of integer numbers is a refinement ring but it is not exchange.
Also it was shown in \cite{ba} that every projective free ring is a refinement ring but it is not necessarily exchange.
 This encourages us to explore the diagonal reduction of regular matrices over refinement rings and extend some results
 in Goodreal's paper \cite{ara}.\\
As we know the ring $M_{m\times n}(R)$ of all $m\times n$ matrices over $R$ is isomorphic to the ring  $Hom_R(nR,mR)$
of all the homomorphisms from $nR$ to $mR$. So we use $Hom_R(nR,mR)$ for the ring $M_{m\times n}(R)$.
Let $R$ be a commutative ring, and let $M$ be a finitely generated projective $R$-module. Let $P$ be a prime ideal of
$R$, and let $R_P$ be the localization of $R$. Then $R_P$ is a local ring, and so $M_P\cong M\bigotimes\limits_{R}R_P$
is a free $R_P$-module. If there exists a fixed $n$ such that
$M_P\cong R_P^n$ for all prime ideals $P$ of $R$, we say that $P$ is a finitely generated projective $R$-module of
constant rank.
\begin{thm} Let $R$ be a commutative refinement ring. Also let $M$ and $N$ be finitely generated projective $R$-modules.
 Then the following statements are equivalent:
\begin{enumerate}
\item [(1)]{\it $M\cong N$.} \item [(2)]{\it $M_P\cong N_P$ for
all prime ideals $P$ of $R$}.
\end{enumerate}
\end{thm}
 \begin{proof}
  $(1)\Rightarrow (2)$ This is obvious.
$(2)\Rightarrow (1)$ Suppose that $M\not\cong N$. In view of
\cite{ba}, there exist orthogonal idempotents $e_1,\cdots ,e_k\in
R$ and non-negative integers $t_{ij}$ such that $M\cong
t_{11}(e_1R)\oplus \cdots \oplus t_{1k}(e_kR)$ and $N\cong
t_{21}(e_1R)\oplus \cdots \oplus t_{2k}(e_kR)$. Then we have some
$e_j\neq 0$ and $t_{1j}\neq t_{2j}$. This shows that there exists
a prime ideal $P$ of $R$ such that $e_j\not\in P$, as prime
radical is nil. For any $i\neq j$, as $e_ie_j=0\in P$, we see that
$e_i\in P$. Thus, $$M_P\cong M\bigotimes_{R}R_P\cong
\bigoplus_{i=1}^{k}t_{1i}(e_iR)\bigotimes _{R}R_P.$$ If $i\neq j$,
then $e_i\in P$, and so $1-e_i\not\in P$. But $(1-e_i)e_i=0$, and
so $\frac{e_i}{1}=0$ in $R_P$. Hence,
$(e_iR)\bigotimes_{R}R_P=\frac{e_i}{1}R_P=0$. Further,
$\frac{e_j}{1}=\frac{1}{1}$ in $R_P$. Therefore $M_P\cong
t_{1j}R_P$. Likewise, $N_P\cong t_{2j}R_P$. As $R$ is commutative,
so is $R_P$, and so $R_P$ has
 Invariant Basis Number.
Thus, $M_P\not\cong N_P$, a contradiction. This completes the proof.
\end{proof}
\begin{cor} Let $R$ be a commutative refinement ring. Then every finitely generated projective $R$-module of
constant rank is free.
\end{cor}
\begin{proof}
Let $M$ be a finitely generated projective $R$-module of constant rank $n$. Then for all prime ideals $P$ of $R$,
we have $M_P\cong R_P^n\cong (R^n)_P$. In light of Theorem 1, $M\cong R^n$, as desired.
\end{proof}
\begin{cor}
 Let $R$ be a commutative refinement ring. Then every stably free $R$-module is free.
 \end{cor}
 \begin{proof}
 This is obvious by Corollary 2.
 \end{proof}
 \begin{thm}
  Let $R$ be
a commutative refinement ring. Let $M$ and $N$ be finitely generated projective $R$-modules.
 Then the following statements are equivalent:
\begin{enumerate}
\item [(1)] $M\cong N$.
\item [(2)] $M_P\cong N_P$ for all maximal ideals $P$ of $R$.
\end{enumerate}
\end{thm}
\begin{proof}
 $(1)\Rightarrow (2)$ This is obvious.
$(2)\Rightarrow (1)$ Suppose that $M\not\cong N$. In view of
\cite{ba}, there exist orthogonal idempotents $e_1,\cdots ,e_k\in
R$ and non-negative integers $t_{ij}$ such that $M\cong
t_{11}(e_1R)\oplus \cdots \oplus t_{1k}(e_kR)$ and $N\cong
t_{21}(e_1R)\oplus \cdots \oplus t_{2k}(e_kR)$. Then we have some
$e_j\neq 0$ and $t_{1j}\neq t_{2j}$. This shows that there exists
a maximal ideal $P$ of $R$ such that $e_j\not\in P$, otherwise,
$e_j$ belongs to all maximal ideals, then $e_j\in J(R)$. This
implies that $e_j=0$, a contradiction. For any $i\neq j$, as
$e_ie_j=0\in P$, we see that $e_i\in P$, as every maximal ideal is
a prime ideal. Similarly to the discussion of Theorem 1, we get a
contradiction as well. This completes the proof.
\end{proof}
Let $R$ be a commutative ring, and let $0\neq x\in R$. Choose
$S=\{ 1,x,x^2,\cdots ,x^n,\cdots \}$.Then $S$ is a multiplicative
closed subset of $R$. Denote $R_{(x)}=S^{-1}R$. Let $f_1,\cdots
,f_n\in R$. Denote $(f_1,\cdots,f_n)$ the ideal generated by $\{
f_1,\cdots ,f_n\rbrace$.
\begin{thm}
Let $R$ be a commutative refinement ring, and let
$R=(f_1,\cdots,f_n)$. Let $M$ and $N$ be finitely generated
projective $R$-modules. Then the following statements are
equivalent:
\begin{enumerate}
\item [(1)]{\it $M\cong N$}. \item [(2)] $M_{(f_i)}\cong
N_{(f_i)}$ for each $i$.
\end{enumerate}
\end{thm}
\begin{proof}
$(1)\Rightarrow (2)$ This is obvious.
$(2)\Rightarrow (1)$ Let $P\in Spec(R)$. Since $R=(f_1,\cdots,f_n)$, we can find some $f_i$, such that
 $f_i\not\in P$. Hence, $f_i\in R-P$. Let $S=\{f_i^n~|~n\geq 0\}, T=R-P$, then $S\subseteq T$
 Let $\psi: R\to S^{-1}R, r\mapsto \frac{r}{1}; \varphi: S^{-1}R\mapsto T^{-1}R, \frac{r}{s}\mapsto \frac{r}{s}; \phi=\varphi\psi: R\to T^{-1}R$. We check that\\
$M_P=T^{-1}R\bigotimes\limits_{\phi}M\cong T^{-1}R\bigotimes\limits_{\varphi\psi}M\cong T^{-1}R\bigotimes\limits_{\psi}S^{-1}M=T^{-1}R\bigotimes\limits_{\psi}M_{(f_i)}$\\ $\cong T^{-1}M_{(f_i)}.$
Likewise, $N_P\cong T^{-1}N_{(f_i)}$. As $M_{(f_i)}\cong N_{(f_i)}$, we see that $M_P\cong N_P$. By using Theorem 1, we get
$M\cong N$, as required.
\end{proof}
\begin{cor} Let $R$ be a commutative refinement ring $a\in R$. Also let $M$ and $N$ be
finitely generated projective $R$-modules. Then the following are equivalent:
\begin{enumerate}
\item [(1)] $M\cong N$.
\item [(2)]{\it $M_{(a)}\cong N_{(a)}$ and $M_{(1-a)}\cong N_{(1-a)}$.}\vspace{-.5mm}
\end{enumerate}
\end{cor}
\begin{proof}
This is obvious by the above theorem , as $aR+(1-a)R=R$.
\end{proof}
\section{Hermite rings and diagonal reduction of matrices over refinement rings}
In this section we investigate some elementary properties of Hermite rings. Some known properties of rings, for example stable range 1 lifts from $\frac{R}{J(R)}$ to $R$ but as we show in the next example, it is not true for Hermite property.
\begin{exam}
Let $R={\Bbb Z_4}[X]$, then $\frac{R}{J(R)}$  is a Hermite ring while $R$ is not.
To see this let $\bar{x}$ be the integer number $x$ modulo $4$. Then $J(R)= \lbrace\bar{0}, \bar{2}\rbrace [X]$, so $\frac{R}{J(R)}\cong {\Bbb Z_2[X]}$ that is a PID and hence  a Hermite ring. It is easily prove that $I=2R+XR$ is a right ideal of $R$ that can not be generated by one element of $R$, that shows $R$ is not B$\acute{e}$zout therefore it is not Hermite ring.
\end{exam}
\begin{prop} Let $R$ be a Hermite ring and $X$ be an variable on $R$ such that $R[[X]]$ is  a B$\acute{e}$zout ring, then $R[[X]]$ is a Hermite ring.
\end{prop}
\begin{proof} Let $\psi : R[[X]]\rightarrow R$ be a homomorphism such that $\psi(f(x))=f(0)$. It is easy to show that $ker(\psi)\subseteq J(R[[X]])$ and $\psi$ is an epimorphism. So $\frac{R[[X]]}{ker\psi} \cong R$ that shows $\frac{R[[X]]}{ker\psi}$ is a Hermite ring. But $R[[X]]$ is a B$\acute{e}$zout ring, then by \cite{zaba} $R[[X]]$ is a Hermite ring.
\end{proof}
In his study of B$\acute{e}$zout ring in 2009 Toganbaev \cite{asgar} proved that for any ring $R$ there exist a B$\acute{e}$zout ring $S$ and an
 idempotent $e\in S$ such that $R\cong eSe$. Then we conclude that for any B$\acute{e}$zout ring $R$ and idempotent $e\in R,$ $eRe$ is not necessarily
  a B$\acute{e}$zout ring. Hence for any Hermite ring $R$, it is an open problem whether $eRe$ is Hermite.\\
\begin{exam} Let $F$ be a field and $X,Y$ be two variables on $F$. Then $F[X]$ and $F[Y]$ are two Hermite $F$-algebra while $F[X] \otimes_F F[Y]$ is not a
 Hermite $F$-algebra.
\end{exam}
Since $F[X]$ and  $F[Y]$ are two principal ideal domains then they are commutative B$\acute{e}$zout domain and by \cite[Theorem 5.2]{kap}  they are Hermite. It is easy to show that $F[X] \otimes_F F[Y]\cong F[X,Y]$  is not a B$\acute{e}$zout ring since $I=(X,Y)$ can not be generated by an element. So it is not a Hermite ring.\\
Let $R$ be a ring and $M$ be an $R-R$-bi-module. Then the trivial
extension $T(R,M)$ is the ring $\{ (r,m)~|~r\in R, m\in M\}$,
where the operations are defined as follows\\ For any $r_1,r_2\in
R, m_1, m_2\in M$,
$$\begin{array}{c}
(r_1,m_1)+(r_2,m_2)=(r_1+r_2, m_2+m_2),\\
(r_1,m_1)(r_2,m_2)=(r_1r_2, r_1m_2 + m_1r_2).
\end{array}$$
It is obvious that, $T(R,M)\cong$ $\lbrace\left(\begin{array}{cc}
a&b\\
0&a
\end{array}
\right)\vert a,\in R, b\in M \rbrace$.\\Now we want to prove that the Hermitian property does not lift from a ring to
its trivial extension. To construct such example we use the notion of FP-injective module.\\
Following \cite{Fr} an $R$-module $M$ is FP-injective if, for each finitely presented $R$-module $N$, $Ext^1_R(N,M)=0$.
It is obvious that $Ext^1_{\Bbb Z}({\Bbb Z_4},{\Bbb Z})\neq 0$, that shows that ${\Bbb Z}$ is not FP-injective.\\
Recall that a ring $R$ is called Coherent ring if, all its finitely generated ideals are finitely presented. Also $R$
is called reduced ring if, it has no non-trivial nilpotent element.
\begin{exam}
Let $R = \lbrace\left(\begin{array}{cc}
a&b\\
0&a
\end{array}
\right)\vert a,b,\in {\Bbb Z} \rbrace$. Then $R$ is not a Hermite ring, while ${\Bbb Z}$ is an elementary divisor ring.
\end{exam}
Since ${\Bbb Z}$ is a principal ideal domain, then it is an elementary divisor ring. Now let $R$ be a Hermite ring.
 As ${\Bbb Z}$ is Coherent reduced ring then by \cite[Corollary 3.3]{Fr}  ${\Bbb Z}$ is a B$\acute{e}$zout,
  FP- injective and all its finitely generated sub-modules are cyclic, that is not true since ${\Bbb Z}$ is not
   FP-injective. Following Chen \cite{Chen} a homomorphism  $f\in Hom_R(nR,mR)$ is
called regular if there exists a homomorphism $g\in Hom_R(mR,nR)$
such that $fgf=f$\\Ara proved the following proposition over
exchange rings \cite[Proposition 2.2]{ara}  and Chen proved it, in
a different way. From those two proofs, we see
 that we can extend it from exchange ring to refinement rings. \cite[Proposition 7.2.2]{Chen}
\begin{prop} Let $R$ be a refinement ring and $f\in M_{m\times n}(R)$ be regular.\\
$(1)$ Suppose that $n\geqslant m$. Then $f$ admits a diagonal reduction if and only if there are decompositions,\\
$ker(f)= K_1\oplus....\oplus K_n$, $imf=I_1\oplus.....\oplus I_m$ and $coker(f)=C_1\oplus....\oplus C_m$ such that
$K_j\oplus I_j\cong C_j\oplus I_j \cong R, j=1,....,m$ and $K_j\cong R, j=m+1,...,n$\\
$(2)$ Suppose that $m\geqslant n$. Then $f$ admits a diagonal reduction if and only if there are decompositions,\\
$ker(f)= K_1\oplus....\oplus K_n$, $imf=I_1\oplus.....\oplus I_n$ and $coker(f)=C_1\oplus....\oplus C_m$
such that $K_j\oplus I_j\cong C_j\oplus I_j \cong R, j=1,....,m$ and $K_j\cong R, j=m+1,...,n$.
\end{prop}
\begin{proof}
The proof is similar to the proof of \cite[Proposition 2.2]{ara}.
\end{proof}
For any finitely generated $R$-modules $K, I, C$  that $K\oplus I\cong nR, I\oplus C\cong mR$, with
$n\geq m$ if we can write $K= K_1\oplus....\oplus K_n$, $I=I_1\oplus...\oplus I_n$ and $C=C_1\oplus....\oplus C_m$
such that $K_j\oplus I_j\cong C_j\oplus I_j \cong R, j=1,...,n$ and $K_j\cong R, j=n+1,...,m$, we say it is a diagonal
refinement for that decompositions.\\
\begin{lem} Let $R$ be a refinement ring. Let $nR\cong K\oplus I$, $mR\cong R\oplus C$ with $n\geq m$ and assume that
 $K\cong K^\prime \oplus X$, $C\cong C^\prime \oplus X$ for some $R-$modules $K^\prime, C^\prime, X$. If the
  above decompositions have a diagonal refinement, then the decomposition  $nR\cong K\oplus I$, $mR\cong R\oplus C$
  have a diagonal reduction.
\end {lem}
\begin{proof} The proof is similar to the proof of \cite[Proposition 2.3]{ara}.
\end{proof}
The following were obtained over the exchange rings and now we extend  them over refinement rings.
\begin{thm} Let $R$ be a refinement ring with this property that $2R\oplus A\cong R\oplus B$ implies that $R\oplus A\cong B$ for all finitely generated projective $R$-modules $A, B$ and $B$ is a generator. Then every regular square matrix over $R$ admits a diagonal reduction.
\end{thm}
\begin{proof} Let $f:nR\rightarrow nR$ be a regular matrix over $R$. Let $K=ker(f), I=im(f), C=coker(f)$. By Lemma 3.3
 we need only to prove that $K, I, C$ can be refined as $K=K_1\oplus......\oplus K_n, I=I_1\oplus....\oplus I_n, C=C_1\oplus....C_n$ such that $K_j\oplus I_j \cong I_j\oplus C_j\cong R, j=1,2,...n$.
 As $f$ is a regular homomorphism then by \cite[Lemma 14.1.1]{Chen}  $K\oplus I\cong I\oplus C\cong nR$.
 As $R$ is a refinement ring and $K\oplus I\cong I\oplus C$ for finitely generated $R$-modules $K, I, C$ then
 there exist $R-$ modules $K_1, K_2, I_1, I_2$ such that $ K\cong K_1\oplus K_2$ $I\cong I_1\oplus I_2,$ and
 $K_1\oplus I_1\cong I, K_2\oplus I_2\simeq C$. By Lemma 3.3, it suffices to find a diagonal refinement for the decompositions $nR\cong K_1\oplus (I\oplus K_2)\cong (I\oplus K_2) \oplus I_2$. Hence we can assume that $K$ is isomorphic to a direct summand of $I$ and $nR$ is isomorphic to a direct summand of $2I$ and therefore $I$ is a generator.\\
 Now we have from $nR\cong K\oplus I\cong I\oplus C\cong (n-1)R\oplus R$ that $nR\oplus C\cong K\oplus I\oplus C\cong (n-1)R\oplus (R\oplus K)$, with $R\oplus K$ a generator and by cancellation property of the assumption in (n-1) times we have $R\oplus K\cong R\oplus C$. As $R$ is a refinement ring and  $R, K, C$ are finitely generated projective $R$- modules then there exist finitely generated projective $R$-modules $R_1, R_2, C_1, C_2$ such that $R\cong R_1\oplus R_2,C_1\oplus C_2\cong C$ and $R_1\oplus C_1\cong R, R_2\oplus C_2\cong K$. By Proposition 3.2. we need only to find a diagonal refinement for the decompositions $R_2\oplus (I\oplus C_2)\cong (I\oplus C_2) \oplus C_1$. By these decompositions we can assume that there exists a finitely generated projective $R$-module $E$ such that $E\oplus K\cong E\oplus C\cong R$. Then we can write $nR\oplus E\cong 2R\oplus (n-2)R\oplus E\cong K\oplus I\oplus E\cong R\oplus I$. Since $I$ is a generator and by hypothesis we can cancel $R$ and get $(n-1)R\oplus E\cong I$. So we have \\
 $I\cong E\oplus R\oplus R\oplus R.....\oplus R$,  $K\cong K\oplus 0\oplus 0\oplus 0.....\oplus 0$,   $C\cong C\oplus 0\oplus 0\oplus 0...\oplus0$. Hence we get the result by Proposition 3.2.
\end{proof}
In the next theorem we want to investigate the diagonalizability
of regular matrices by the following cancellative property that is
proved by Ara \cite{ara} over exchange ring and can be extended to
refinement ring.
\begin{thm}
Let $R$ be a refinement ring. Then every $m\times n$ regular matrix over $R$ admits a diagonal reduction if and only if $2R\oplus A\cong R\oplus B$ implies that $R\oplus A\cong B$ for all finitely generated projective $R$-modules $A, B$.
\end{thm}
As we see in many examples the diagonalizability property can not lift from $\frac{R}{J(R)}$ to $R$ but we show in the next Proposition that it is possible for regular matrices.
\begin{prop}
Let $R$ be a refinement ring such that every $m\times n$ regular matrix over $\frac{R}{J(R)}$ admits a diagonal reduction. Then every regular matrix over $R$ admits a diagonal reduction.
\end{prop}
\begin{proof}
Let $2R\oplus A\cong R\oplus B$ for finitely generated projective $R$-modules $A, B$. Then we have
 $$  2\frac{R}{J(R)} \oplus \frac{A}{RJ(A)} \cong \frac{2R\oplus A}{J(2R\oplus A)}\cong \frac{R\oplus B}{J(R\oplus B)} \cong \frac{R}{J(R)}\oplus \frac{B}{RJ(B)}.$$
 Since every $m\times n$ regular matrix over $\frac{R}{J(R)}$ admits a diagonal reduction then by Theorem 3.5
  we deduce that $$\frac{R}{J(R)}\oplus \frac{A}{RJ(A)} \cong \frac{B}{RJ(B)}.$$ Then we have,
  $$\frac{R \oplus A}{J(R\oplus A)}\cong \frac{B}{RJ(B)}\Longrightarrow R\oplus A\cong B.$$ Since $R, A, B$ are finitely generated projective $R$-modules. Therefore by Theorem 3.5. every $m\times n$ regular matrix over $R$ admits a diagonal reduction.
\end{proof}
\begin{cor} A regular ring $R$ is an elementary divisor ring if and only if $\frac{R}{J(R)}$ is an elementary divisor ring.
\end{cor}
\begin{proof}
As every homomorphic image of an elementary divisor ring is an elementary divisor ring, then one direction is obvious. Conversely assume that $R$ is a regular ring and $\frac{R}{J(R)}$ is an elementary divisor ring, by Proposition 3.6. every $m\times n$ regular matrix over $R$ admits a diagonal reduction, but since $R$ is regular it is obvious that every matrix over $R$ is regular. Then $R$ is an elementary divisor ring.
\end{proof}

\end{document}